\documentclass[12pt]{article}
\usepackage{amsthm}
\usepackage{amsmath,amssymb}
\usepackage{graphicx}
\usepackage[margin=2.5cm]{geometry}
\usepackage{graphicx}
\usepackage[dvipsnames]{color}

\newtheorem{theorem}{Theorem}[section]
\theoremstyle{plain}

\newtheorem{lemma}[theorem]{Lemma}
\newtheorem{proposition}[theorem]{Proposition}
\newtheorem{corollary}[theorem]{Corollary}

\newtheorem{problem}[theorem]{Problem}
\newtheorem{property}[theorem]{Property}
 
\theoremstyle{definition}

\definecolor{blau}{rgb}{0.1,0.0,0.9}
\definecolor{gruen}{cmyk}{1.0,0.2,0.7,0.07}
\definecolor{mag}{cmyk}{0.0,0.9,0.3,0.0}

\begin{document}

\title{Restricted extension of sparse partial edge colorings of hypercubes}

\author{Carl Johan Casselgren, Klas Markstr\"om and Lan Anh Pham }
\maketitle

\begin{abstract}
		We consider the following type of question: 
		Given a partial proper $d$-edge coloring 
		of the $d$-dimensional hypercube $Q_d$, and lists of allowed colors
		for the non-colored edges of $Q_d$,
	can we extend the partial coloring to a proper 
	$d$-edge coloring using only 
	colors from the lists?     
	We prove that this question has a positive answer in the case when 
	both the partial coloring and the color lists
	satisfy certain sparsity conditions.
\end{abstract}


\section{Introduction}
The chromatic index $\chi'(G)$ of a (simple) graph $G$ is far simpler 
in terms of its possible values than the chromatic number;
Vizing's theorem \cite{vizing1964estimate} tells us that in order to properly color 
the edges of $G$ we need either 
 $\Delta(G)$ or $\Delta(G)+1$ colors, 
where $\Delta(G)$ denotes the maximum degree of $G$, and
by K\"onig's edge coloring theorem, $\chi'(G) = \Delta(G)$ if $G$ 
is bipartite \cite{konig1916graphen}.
This simplicity quickly disappears in many of the  natural variations on the basic 
edge coloring problem, e.g. the precoloring extension problem, 
where some of the edges of a graph have been (properly)
colored and we want to determine if this
partial coloring can be extended to a proper
edge coloring of the full graph using no extra colors;
indeed this problem
is NP-complete already for $3$-regular bipartite graphs \cite{fiala2003np}.

One of the earlier references explicitly discussing the problem of 
extending a partial edge coloring is \cite{MS90}; there a 
simple necessary condition for the existence of an extension is 
given and the authors  find a class  of graphs where this 
condition is also sufficient.   More recently the question of 
extending a precoloring where the precolored edges form 
a matching has gathered interest; in \cite{EGHKPS} a 
number of positive results and conjectures are given. 
In particular, it is conjectured that for every graph $G$,
if $\varphi$ is an edge precoloring of a matching $M$ in $G$
using $\Delta(G)+1$ colors,
and any two edges in $M$ 
are at distance at least $2$ from each other,  then $\varphi$ 
can be extended to a proper $(\Delta(G)+1)$-edge coloring of $G$;
this was first 
conjectured in \cite{AM2001}, but then with 
distance $3$ instead. By the {\em distance} between two
edges $e$ and $e'$ here we mean the number of edges in a 
shortest path between an endpoint of $e$ and an endpoint of $e'$; 
a {\em distance-$t$ matching} is a matching where any
two edges are at distance at least $t$ from each other.
The {\em $t$-neighborhood} of an edge $e$ is
the graph induced by all edges of distance at most $t$ from $e$.

Note that the conjecture in \cite{EGHKPS} on distance-$2$ matchings
is sharp both with respect to the
distance between precolored edges,
and in the sense that $\Delta(G)+1$ can in general 
not be replaced by $\Delta(G)$,
even if any two precolored edges are at arbitrarily 
large distance from each other \cite{EGHKPS}.
In \cite{EGHKPS}, it is proved that this conjecture hold
for e.g. bipartite multigraphs and subcubic multigraphs, and
in \cite{GK} it is proved that a version of the conjecture with the distance increased to 9 holds for general graphs.

However, for one specific family of graphs,
the balanced complete
bipartite graphs $K_{n,n}$, the edge precoloring
extension problem was studied far earlier than  in the
above-mentioned references.  Here
the  extension problem corresponds to asking whether a
partial latin square can be completed to a latin square.   
In this form the problem appeared already in 1960, when Evans 
\cite{Ev60}  stated his now classic  conjecture that for
every positive integer $n$, if 
$n-1$ edges in $K_{n,n}$ have been (properly) colored, 
then this partial coloring can be extended to 
a proper $n$-edge-coloring of $K_{n,n}$.
This conjecture was
solved for large 
$n$ by H\"aggkvist  \cite{H78} and later for all $n$ by Smetaniuk 
\cite{S81}, and independently by 
Andersen and Hilton \cite{andersen1983thank}.
Generalizing this problem, Daykin and H\"aggkvist \cite{DH}
proved several results on
extending partial edge colorings of $K_{n,n}$, and they also conjectured that  much denser partial colorings  can be extended, as long as the 
colored edges are spread out in a specific sense:
a partial $n$-edge coloring of $K_{n,n}$ is 
{\em $\epsilon$-dense} if there are at most $\epsilon n$ colored edges 
from $\{1,\dots,n\}$
at any vertex and each color in $\{1,\dots,n\}$
is used at most $\epsilon n$ times
in the partial coloring. 
Daykin and H\"aggkvist \cite{DH} conjectured
that for every positive integer $n$,
every $\frac{1}{4}$-dense partial proper $n$-edge
coloring can be extended
to a proper $n$-edge coloring of $K_{n,n}$,
and proved a version of the conjecture for $\epsilon=o(1)$ 
(as $n \to \infty$) and $n$ divisible by 16.
Bartlett \cite{B2013} proved  that this conjecture holds for 
a fixed positive $\epsilon$, and recently a 
different proof which improves the value of $\epsilon$ was 
given in \cite{BKLOT}.

For general edge colorings of  balanced complete
bipartite graphs, Dinitz conjectured,
and Galvin proved \cite{Ga95}, that if each edge of $K_{n,n}$
is given a list of $n$ colors, then there is a proper edge 
coloring of $K_{n,n}$ with support in the lists. 
Indeed, Galvin's result was a complete solution
of the well-known List Coloring Conjecture
for the case of bipartite multigraphs
(see e.g. \cite{haggkvist1992some} for more background on this conjecture and its relation
to the Dinitz' conjecture).

Motivated by the Dinitz' problem,  H\"aggkvist \cite{Ha89} 
introduced the notion of {\em $\beta n$-arrays},
which correspond to list assignments $L$ of forbidden colors for 
$E(K_{n,n})$,
such that each edge $e$ of $K_{n,n}$ is assigned a list $L(e)$
of at most $\beta n$ forbidden colors from $\{1,\dots,n\}$, 
and at every vertex $v$ each color is forbidden on 
at most $\beta n$ edges adjacent to $v$;
we call 
such a list assignment for $K_{n,n}$ {\em $\beta$-sparse}. 
If $L$ is a list assignment for $E(K_{n,n})$, then
a proper $n$-edge coloring $\varphi$ of $K_{n,n}$ {\em avoids}
the list assignment $L$ if $\varphi(e) \notin L(e)$ for 
every edge $e$ of $K_{n,n}$;
if such a coloring exists, then $L$ is {\em avoidable}.
H\"aggkvist conjectured that there exists a fixed 
$\beta>0$, in fact also that $\beta=\frac{1}{3}$,  such that
for every positive integer $n$,
every $\beta$-sparse list assignment for $K_{n,n}$ is avoidable.
That such a $\beta>0$ exists was proved
for 
even $n$
by Andr\'en  in her PhD thesis \cite{andren2010latin},
and later for all 
$n$ in \cite{ACO}.

Combining the notions of extending a sparse precoloring and 
avoiding a sparse list assignment,
Andr\'en et al. \cite{ACM} proved that 
there are constants $\alpha> 0$ and $\beta> 0$, 
such that for every positive integer $n$,
every $\alpha$-dense partial edge
coloring of $K_{n,n}$ can be extended 
to a proper $n$-edge-coloring avoiding any given $\beta$-sparse
list assignment $L$, provided that no edge $e$ is precolored by a color
that appears in $L(e)$.
In contrast to this, it was proved in \cite{EGHKPS} that there are
bipartite graphs $G$
with a precolored matching of size $2$, which is not extendable
to a proper $\Delta(G)$-edge coloring.
These examples have edge densities converging to some constant 
$0<c\leq\frac{1}{2}$, and many of the 
proof methods used in the papers mentioned
above rely on the high edge density of the complete bipartite graph. 
It is thus natural to ask if the good behaviour seen for $K_{n,n}$ will hold for 
well-structured graphs of 
lower densities.

The aim of this paper is to show that some generalizations of this type are possible.   
We will demonstrate that 
results similar to those from  
\cite{ACM}  hold for the family of
$d$-dimensional hypercubes $Q_d$; 
these graphs have a degree which is 
logarithmic in the number of vertices, rather than linear, 
so we are now looking at a
family of graphs with vanishing asymptotic density.   
Instead of bounding the global density of precolored edges,
as in the case of complete bipartite graphs, for
hypercubes
we shall bound the number
of precolored edges appearing in neighborhoods of given size.
Our results are stated in terms of $27$-neighborhoods; this size
of neighborhoods is solely due to proof technical reasons.

To state our main theorem, we need some terminology:
a {\em dimensional matching} $M$ of $Q_d$ is a perfect matching of $Q_d$
such that $Q_d- M$ is isomorphic to two copies of $Q_{d-1}$; evidently
there are precisely $d$ dimensional matchings in $Q_d$.
An edge precoloring of $Q_d$ with colors $1,\dots,d$
is called {\em $\alpha$-dense} if

\begin{itemize}
	
	\item[(i)] there are at most $\alpha d$ precolored edges at each vertex;
	
	\item[(ii)] for every $27$-neighborhood $W$ of an edge $e$ of
	$Q_d$, there are
	at most $\alpha d$ precolored edges with color $i$
	in $W$, $i =1,\dots, d$;
	
	\item[(iii)] for every $27$-neighborhood $W$, and every dimensional
	matching $M$, at most $\alpha d$ edges of $M$ are precolored in $W$.

\end{itemize}
	Here, and in the following, all $t$-neighborhoods are taken with
	respect to edges.
	A list assignment $L$ for $E(Q_d)$ is {\em $\beta$-sparse}
	if the list of each edge is a (possibly empty) subset of $\{1,\dots,d\}$, 
	and
	\begin{itemize}
	
	\item[(i)] $|L(e)| \leq \beta d$ for each edge $e \in E(Q_d)$;
	
	\item[(ii)] for every vertex $v \in V(Q_d)$,
	each color in $\{1,\dots,d\}$ occurs in at most $\beta d$
	lists of edges incident to $v$;
	
	\item[(iii)] for every $27$-neighborhood $W$, and every 
	dimensional matching $M$, any color appears at most $\beta d$
	times in lists of edges of $M$ contained in $W$.
\end{itemize}

Our main result is the following.

\begin{theorem}
\label{mainth}
	There are constants $\alpha> 0$  and $\beta> 0$ such that  
	for every positive integer $d$, if $\varphi$  
	is an $\alpha$-dense $d$-edge precoloring of $Q_d$, $L$ a
	$\beta$-sparse list assignment for $Q_d$,
	and $\varphi(e) \notin L(e)$ for every edge $e \in E(Q_d)$,
	then there is a proper
	$d$-edge coloring of $Q_d$ which agrees with $\varphi$
	on any precolored edge and which avoids $L$.
\end{theorem}


As a corollary of our main theorem we 
note that a version of the conjecture on
precolored distance-$2$ matchings from \cite{EGHKPS},
with $\Delta(G)$ in place of $\Delta(G)+1$,
but with a weaker distance requirement,
 holds for the family of hypercubes.
\begin{corollary}
	There is 
	a constant $\beta>0$ such that
	if $L$ is a $\beta$-sparse list assignment $L$ for $Q_d$
	and $\varphi$ is a $d$-edge precoloring of a distance-$t$ matching
	in $Q_d$, where $t > 55$, then $\varphi$ can be extended
	to a proper $d$-edge-coloring of $Q_d$ which avoids $L$.
\end{corollary}

This follows from the fact that a precolored
distance-$t$ matching
is an $\alpha$-dense precoloring if $t > 55$.
A precolored matching has a much more
restricted structure than a general $\alpha$-dense $d$-edge precoloring, 
so this corollary is most likely far from optimal in terms of 
the lower bound on $t$;
it would be interesting to see 
how far this can be improved. 

If we place both the precolored edges and those with a list of forbidden colors on them  on a matching, then our proof method in fact
trivially yields the following.

\begin{theorem}
\label{secondth}
	Let $\varphi$
	be a $d$-edge precoloring and $L$ a list assignment for
	the edges of $Q_d$. If every edge $e$ which is either
	precolored or satisfies $L(e) \neq \emptyset$ belongs
	to a distance-3 matching $M$ 
	in $Q_d$, then there is a $d$-edge coloring which agrees with $\varphi$
	on any precolored edge, and which avoids $L$.
\end{theorem}
This proves that a slightly stronger version, with $d$ rather than $d+1$ colors,  of the  earlier mentioned conjecture from \cite{AM2001}  holds for the family of hypercubes.

The rest of the paper is organized as follows.
In Section 2 we introduce some terminology and notation
and also outline the proof of Theorem \ref{mainth}.
Section 3 contains the proof of a slightly reformulated version
of Theorem \ref{mainth}; we also indicate how Theorem \ref{secondth}
can be deduced from the proof of Theorem \ref{mainth}.
In Section 4 we give some concluding remarks; in
particular, we give an example
indicating what numerical values of $\alpha$ and $\beta$
in Theorem \ref{mainth} might be best
possible.
At the beginning of Section 3 we shall present
numerical values of $\alpha$ and $\beta$ for which our main
theorem holds, provided that $n$ is large enough.
Finally, throughout the paper, the base of the natural
logarithm is denoted by $e$.

\section{Terminology, notation and proof outline}

	Given an edge precoloring $\varphi$ (or just precoloring, or 
	partial edge coloring)
	of a graph $G$ with $\Delta(G)$ colors, 
	an {\em extension}
	of $\varphi$ is a proper $\Delta(G)$-edge coloring of $G$ which 
	agrees with $\varphi$
	on every precolored edge; if such a coloring of $G$ exists, then
	$\varphi$ is {\em extendable}.



	For a vertex $u \in Q_d$,
	we denote by $E_u$ the set of edges with one endpoint being $u$,
	and for a (partial) edge coloring $f$ of $Q_d$, let $f(u)$
	denote the set of colors on edges in $E_u$ under $f$.
	If two edges $xy$ and $zt$ of $Q_d$ are in a cycle of length $4$
	in $Q_d$, and all the vertices $x,y,z,t$ are distinct,
	then the edges $xy$ and $zt$ are {\em parallel}.

	As noted above, $Q_d$ decomposes into precisely $d$ dimensional
	matchings. 
	Note that a dimensional matching in $Q_d$ contains 
	precisely $2^{d-1}$ edges.
	For a $d$-edge coloring $f$ of $Q_d$, a dimensional matching
	$M$ is called {\em a dimensional matching of color $c$ (under $f$)} if
	there are more edges of $M$ colored
	$c$ than by any other color.

	For the proof of Theorem \ref{mainth}, we shall use
	the {\em standard} $d$-edge coloring $h$ of $Q_d$
	where all edges of the $i$th dimensional matching in $Q_d$
	is colored $i$, $i=1,\dots,d$. A cycle is
	{\em $2$-colored}
	if its edges are colored by two distinct colors
	from $\{1,\dots,d\}$.
	The following property is crucial for
	our proof of Theorem \ref{mainth}.
	
	\begin{property}
	\label{prop:stand}
		In the standard $d$-edge coloring $h$, every edge of $Q_d$
		is in exactly $d-1$ $2$-colored $4$-cycles.
	\end{property}

	Let $\varphi$ be an $\alpha$-dense precoloring of $Q_d$.
	Edges of $Q_d$ which are colored under $\varphi$,
	are called {\em prescribed (with respect to $\varphi$)}.
	For the edge coloring $h$ (or an edge coloring obtained from $h$),
	an edge $e$ of $Q_d$ is called 
	{\em requested (under $h$ with respect to $\varphi$)} 
	if $h(e) = c$ and
	$e$ is adjacent to an edge $e'$ such that $\varphi(e')=c$.
	
	Consider a $\beta$-sparse list assignment $L$ for $Q_d$.
	For the edge coloring $h$ (or an edge coloring obtained from $h$),
	an edge $e$ of $Q_d$ is called a \textit{conflict edge (of $h$ with 
	respect to $L$)} 
	if $h(e) \in L(e)$.
	An \textit{allowed cycle (under $h$ with respect to $L$)} of $Q_d$ is a
	$4$-cycle
	$\mathcal{C}=uvztu$ in $Q_d$ that is $2$-colored
	under $h$, and
	such that
	interchanging colors on $\mathcal{C}$
	yields a proper $d$-edge coloring
	$h_1$ of $Q_d$
	where none of $uv$, $vz$, 
	$zt$, $tu$ is a conflict edge.
	We call such an interchange {\em a swap in $h$}.
	\vspace{0.5cm}

Instead of proving Theorem \ref{mainth} we shall in fact prove
the following theorem, which is easily seen
so imply Theorem \ref{mainth}.

\begin{theorem}
\label{th:main2}
	There are constants $\alpha>0$, $\beta>0$ and $d_0$, such that  
	for every positive integer $d \geq d_0$, if $\varphi$
	is an $\alpha$-dense precoloring of $Q_d$
	and $L$ a $\beta$-sparse list assignment for $Q_d$,
	and $\varphi(e) \notin L(e)$ for every edge $e \in E(Q_d)$,
	then there is an extension of
	$\varphi$ which avoids $L$.
\end{theorem}

Below we outline the proof of Theorem \ref{th:main2}.
Let $h$ be the standard proper $d$-edge coloring of $Q_d$
defined above, $\varphi$ an $\alpha$-dense precoloring
of $Q_d$, and $L$ a $\beta$-sparse list assignment for $E(Q_d)$.

%
%

\begin{enumerate}

	\item[Step I.] Given the standard $d$-edge coloring $h$
	of $Q_d$, find a permutation $\rho$ of the elements of the
	set $\{1,\dots,d\}$
	such that in the proper $d$-edge-coloring $h'$ obtained
	by applying $\rho$ to the colors used in $h$, 
	locally, each dimensional matching
	in $Q_d$ contains ``sufficiently few'' conflict edges with $L$,
	as well as ``sufficiently few'' requested edges with respect to $\varphi$.
	Moreover, we require that each vertex $u$ of $Q_d$ satisfies
	that $E_u$ contains ``sufficiently few'' conflict and requested edges,
	and that each edge of $Q_d$ belongs to ``many'' allowed cycles
	under $h'$. These conditions shall be more precisely articulated below.

	\item[Step II.] From the precoloring $\varphi$ of $Q_d$,
	define a new edge precoloring $\varphi'$ such that an
	edge $e$ of $Q_d$ is colored under $\varphi'$ if and only if
	$e$ is colored under $\varphi$ or $e$ is a conflict edge
	of $h'$ with respect to $L$.
	We shall also require that 
	locally,
	each of the colors in $\{1,\dots,d\}$
	is used a bounded number of times under $\varphi'$.
	
	
	\item[Step III.]
	From $h'$, construct a proper $d$-edge coloring $h''$
	of $Q_d$ such that under $h''$, no edge in $Q_d$ is both requested
	and prescribed (with respect to $\varphi'$); this is done by
	swapping on a set of disjoint allowed $4$-cycles.
	We also require that each requested
	edge $e$ of $h''$ is adjacent to at most one edge $e'$
	such that $h''(e) = \varphi'(e')$.
	
	\item[Step IV.] For each edge $e$ of $Q_d$ that is prescribed
	with respect to $\varphi'$, construct a subset $T_e \subseteq E(Q_d)$,
	such that performing a series of swaps in $h''$ on allowed cycles,
	all edges of which are in $T_e$, yields a coloring $h''_1$ where
	$h''_1(e) = \varphi'(e)$. Moreover, if $e$ and $e'$ are prescribed
	edges of $Q_d$, then the sets $T_{e'}$ and $T_e$ will be disjoint.
	Thus, performing the series of swaps on all the sets $T_e$
	associated with prescribed edges $e$ yields a proper $d$-edge coloring $\hat h$
	which is an extension of $\varphi'$ (and thus $\varphi$), 
	and which avoids $L$.
	
	
\end{enumerate}


\section{Proofs}

In this section we prove Theorem \ref{th:main2}. In the proof we shall verify
that it is possible to perform Steps I-IV described above
to obtain a proper $d$-edge-coloring of $Q_d$ that is an extension of $\varphi$
and which avoids $L$. This is done by proving a lemma in each step.

We will not specify the value of $d_0$ in the proof but rather
assume that $d$ is large enough whenever necessary. Since the proof will contain
a finite number of inequalities that are valid if $d$ is large enough, this
suffices for proving Theorem \ref{th:main2}.

The proof of Theorem \ref{th:main2} involves a number of 
functions and parameters:
$$\alpha, \beta, \gamma, \kappa, \epsilon, \epsilon_0, \tau,$$
and a number of inequalities that they must satisfy. For the reader's convenience,
explicit choices for which the proof holds are presented here:
$$\alpha = 10^{-622}, \beta= 2 \cdot 10 ^{-622}, \gamma=2^{-11},
\kappa= 9/2^{11} $$
$$\epsilon =  2^{-3}, \epsilon_0=2^{-8}, \tau=2^{-7}.$$

We remark that since the numerical values of $\alpha$ and $\beta$ are not anywhere near
what we expect to be optimal, we have not put an effort into choosing optimal values
for these parameters. 
Finally, for simplicity of notation,
we shall omit floor and ceiling signs whenever these are not crucial.

\begin{proof}[Proof of Theorem \ref{th:main2}]
Let $\varphi$ be an $\alpha$-dense precoloring of $Q_d$, and let $L$
be a $\beta$-sparse list assignment for $Q_d$. Moreover, let $h$ be
the standard $d$-edge coloring defined above. 

\bigskip

\noindent
{\bf Step I:} We use the following lemma
for constructing a required $d$-edge-coloring $h'$ from $h$.
\begin{lemma}
	\label{alpha}
          Let $\gamma, \tau <1$ be 
					constants such that $0< \alpha, \beta \leq \gamma$, 
         $2^{\frac{1}{ \gamma}+1}e \alpha < \dfrac{\gamma}{3}$, 
         $2^{\frac{1}{ \gamma}}e \beta <  \dfrac{\gamma}{3}$
         and $2^{\frac{2}{\tau-2\beta} +1}e \beta<\tau-2\beta$.
        There is a permutation $\rho$ of $\{1,\dots,d\}$,
				such that applying $\rho$ to
				the set of colors $\{1,\dots,d\}$ 
				used in $h$, we obtain a $d$-edge coloring $h'$ of $Q_d$
				satisfying the following:
        \begin{itemize}
	\item[(a)] For every $26$-neighborhood 
	$W$, and every dimensional
	matching $M$, 
	at most $\gamma d$ edges of $M \cap E(W)$ are requested.
	
	\item[(b)] For every $27$-neighborhood $W$, and every dimensional
	matching $M$, at most $\gamma d$ edges of $M \cap E(W)$ are conflict.
	
	\item[(c)] No vertex $u$ in $Q_d$ satisfies that $E_u$ contains 
	more than $\gamma d$ requested edges. 
	
	\item[(d)] No vertex $u$ in $Q_d$ satisfies that $E_u$ contains 
	more than $\gamma d$ conflict edges. 
	
	\item[(e)] Each edge in $Q_d$ belongs to at least $(1-\tau)d$ 
	allowed cycles.
	
	\end{itemize}
	\end{lemma}
	
	\begin{proof}
	Let $A$, $B$, $C$, $D$ and $E$ be the number of permutations which do not 
	fulfill the conditions $(a)$, $(b)$, $(c)$, $(d)$ and $(e)$, respectively. 
	Let $X$ be the number of 
	permutations satisfying the five conditions 
	$(a)$, $(b)$, $(c)$, $(d)$ and $(e)$.
	There are $d!$ ways to permute the colors, so we have
	$$X \geq d! - A- B - C - D - E$$ 
	We will now prove that $X$ is greater than $0$.
	
	\begin{itemize}
	\item Recall that an edge $e$ is requested 
	if $e$ is adjacent to an edge $e'$ such that
	$h(e)=\varphi(e')$. Let $M'$ be a dimensional matching,
	and consider a subset $M \subseteq M'$ of all edges in $M'$ that are
	contained in a given
	$26$-neighborhood $W_1$. Then every edge of $M$ and 
	every edge adjacent to an edge of $M$ is contained in a
	$27$-neighborhood $W_2$ containing $W_1$. 
	Since all edges in $M$ have the same color in any
	edge coloring obtained from $h$ by permuting colors, and 
	there are
	at most $\alpha d$ precolored edges with color $i$
	in $W_2$, $i =1,\dots, d$, the maximum number of requested 
	edges in $M$
	is $\alpha d$. In other words, 
	no subset of a dimensional matching contained in a $26$-neighborhood
	contains more than $\alpha d$ requested edges.
	Since $\gamma \geq \alpha$, this means that
	all permutations satisfy condition $(a)$ or $A=0$.
	
	\item Since all edges that are in the same dimensional 
	matching have the same color under $h$ and 
	for every $27$-neighborhood $W$, and every 
	dimensional matching $M$, any color appears at most $\beta d$
	times in lists of edges of $M$ contained in $W$, 
	we have that the maximum number of conflict edges in 
	a subset of a given dimensional matching contained in a 
	$27$-neighborhood
	is $\beta d$. 
	Since $\gamma \geq \beta$, this means that
	all permutations satisfy condition $(b)$ or $B=0$.
		
	\item To estimate $C$, let $u$ be a fixed vertex of $Q_d$,
	and let $S$ be a set of size $\gamma d$ 
	of edges of $E_u$. 
	There are $d \choose \gamma d$ ways to choose $S$. For a 
	vertex $v$ adjacent to $u$, 
	if $uv$ is a requested edge, then the colors used in $h$
	should be permuted in such a way that in the resulting coloring $h'$,
	$uv$ is colored by some color in the set
	$\{(\varphi(u) \cup \varphi(v)) \setminus \varphi(uv)\}$.
	Since $|\varphi(u)| \leq \alpha d$ 
	and $|\varphi(v)| \leq \alpha d$, there are at 
	most $(2 \alpha d)^{\gamma d}$
	ways to choose which colors from $1,2,\dots,d$ to assign to
	the edges in $S$ so that all
	edges in $S$ are requested. The rest of the colors can be arranged in any 
	of the $(d - \gamma d)!$ possible ways. In total this gives at most
	$${d \choose \gamma d}(2\alpha d)^{\gamma d}(d-\gamma d)! 
	= \dfrac{d!(2\alpha d)^{\gamma d}}{(\gamma d)!}$$
	permutations that do not satisfy condition $(c)$ on vertex $u$.
	
	There are $2^d$ vertices in $Q_d$, so we have
	$$C \leq 2^d \dfrac{d!(2\alpha d)^{\gamma d}}{(\gamma d)!}$$
	
	\item  To estimate $D$, let $u$ be a fixed vertex of $Q_d$, and
	let $S$ be a set of size $\gamma d$ ($|S|=\gamma d$) 
	of edges from $E_u$.
	For a vertex $v$ adjacent to $u$, if $uv$ is a conflict edge,
	then the colors used in $h$
	should be permuted in such a way that in the resulting coloring $h'$,
	the color of $uv$ is in $L(uv)$.
	Since $|L(uv)| \leq \beta d$, there are at most $(\beta d)^{\gamma d}$
	ways to choose which colors from $\{1,2,\dots,d\}$
	to assign to the edges in $S$ so that all
	edges in $S$ are conflict. The rest of the colors can be arranged in any 
	of the $(d - \gamma d)!$ possible ways. In total this gives at most
	$${d \choose \gamma d}(\beta d)^{\gamma d}(d-\gamma d)! 
	= \dfrac{d!(\beta d)^{\gamma d}}{(\gamma d)!}$$
	permutations that do not satisfy condition $(d)$ on vertex $u$.
	There are $2^d$ vertices in $Q_d$, so we have
	$$D \leq 2^d \dfrac{d!(\beta d)^{\gamma d}}{(\gamma d)!}$$
	
	\item To estimate $E$, let $uv$ be a fixed edge of $Q_d$. 
	Each cycle $\mathcal{C} = uvztu$ 
	containing $uv$ is uniquely defined by an edge 
	$zt$ which is parallel with $uv$. Moreover, a permutation $\varsigma$ 
	is in $E$ if 
	and only if there are more than $\tau d$ choices 
	for $zt$ so that 
	$\mathcal{C}$ is not allowed. We shall count the number of 
	ways $\varsigma$ could be 
	constructed for this to happen. First, note that for each 
	choice of color $c_1$ from $\{1,\dots, d\}$, for the
	dimensional matching which contains $uv$, there are up 
	to $2\beta d$ cycles that
	are not allowed because of this choice. This follows from the fact
	that there are at most $\beta d$ choices
	for $t$ (or $z$) such that $L(ut)$ (or $L(vz)$) contains $c_1$. 
	So for a permutation 
	$\varsigma$ to belong to $E$, $\varsigma$ must satisfy that 
	at least $(\tau-2\beta)d$ cycles
	containing $uv$ are forbidden because of 
	the color assigned to the dimensional matching containing $ut$ and
	$vz$.
	
	Let $S$ be a set of edges, $|S|=(\tau-2\beta) d$, such that for
	every edge $zt \in S$,
	the cycle $\mathcal{C}= uvztu$ is not allowed because
	of colors assigned to
	$ut$
	and $vz$. 
	There are ${d-1} \choose {(\tau - 2\beta)d}$ ways to choose $S$.
	Furthermore, $L(uv)$ and $L(zt)$ contain at most $\beta d$ colors
	each, so there are at most
	$2\beta d$ choices for a color for the dimensional matching
	containing $ut$ and $vz$ 
	that would make $\mathcal{C}$ disallowed because
	of the color assigned to this dimensional matching.
	The remaining colors can be permuted in 
	$(d-1-(\tau-2\beta)d)!=((1-\tau+2\beta)d-1)!$ ways. 
	
	Hence, the total number of permutations $\sigma$ with not 
	enough allowed cycles for 
	a given edge is bounded from above by
	$$d{{d-1} \choose {(\tau - 2\beta)d}}{(2\beta d)^{(\tau - 2\beta)d}}
	{((1-\tau+2\beta)d-1)!}
	= \dfrac{d!{(2\beta d)^{(\tau - 2\beta)d}}}{((\tau - 2\beta)d)!}$$
	and the total number of permutation $\sigma$ that have too 
	few allowed cycles for at 
	least one edge is bounded from above by
	$$ 2^{d-1}d\dfrac{d!{(2\beta d)^{(\tau - 2\beta)d}}}
	{((\tau - 2\beta)d)!}$$
	\end{itemize}

	Hence,
	$$X \geq d! - 2^d \dfrac{d!(2\alpha d)^{\gamma d}}{(\gamma d)!}
	- 2^d \dfrac{d!(\beta d)^{\gamma d}}{(\gamma d)!}
	- 2^{d-1}d\dfrac{d!{(2\beta d)^{(\tau - 2\beta)d}}}
	{((\tau - 2\beta)d)!}$$
	
	Using Stirling's approximation, $n! \geq n^ne^{-n}$ and $2^{d-1}d <\dfrac{2^{2d}}{3}$, we have
	$$X \geq d!\Big(1-2^d \dfrac{e^{\gamma d}(2\alpha d)^{\gamma d}}{(\gamma d)^{\gamma d}}
	- 2^d \dfrac{e^{\gamma d}(\beta d)^{\gamma d}}{(\gamma d)^{\gamma d}}
	- 2^{2d}\dfrac{e^{(\tau - 2\beta)d}{(2\beta d)^{(\tau - 2\beta)d}}}
	{3((\tau - 2\beta)d)^{(\tau - 2\beta)d}}\Big)$$
	$$X > d!\Big(1-\big(\dfrac{2^{\frac{1}{\gamma}+1}e \alpha}{\gamma}\big)^{\gamma d} -
	\big(\dfrac{2^{\frac{1}{\gamma}}e \beta}{\gamma}\big)^{\gamma d}
	- \dfrac{1}{3}\big(\dfrac{2^{\frac{2}{\tau-2\beta} +1}e \beta}
	{\tau-2\beta}\big)^{\tau-2\beta d} \Big)$$
	
	Using the conditions 
	$\dfrac{2^{\frac{1}{\gamma}+1}e \alpha}{\gamma} <\dfrac{1}{3}$,
	$\dfrac{2^{\frac{1}{\gamma}}e \beta}{\gamma} <\dfrac{1}{3}$, and
	$\dfrac{2^{\frac{2}{\tau-2\beta} +1}e \beta}{\tau-2\beta}<1$, we have
	
	$$\big(\dfrac{2^{\frac{1}{\gamma}+1}e \alpha}
	{\gamma}\big)^{\gamma d} < \dfrac{1}{3}
	\hspace{0.2cm} \mbox{and} \hspace{0.2cm} 
	\big(\dfrac{2^{\frac{1}{\gamma}}e \beta}{\gamma}\big)^{\gamma d} < 
	\dfrac{1}{3}
	\hspace{0.2cm} \mbox{and} \hspace{0.2cm} 
	\dfrac{1}{3}\big(\dfrac{2^{\frac{2}{\tau-2\beta} +1}e \beta}
	{\tau-2\beta}\big)^{\tau-2\beta d} <\dfrac{1}{3}$$
	
	This implies $X >0$.

	\end{proof}

\bigskip

\noindent
{\bf Step II:} Let $h'$ be the proper $d$-edge coloring satisfying conditions
(a)-(e) of Lemma \ref{alpha} obtained in the previous step.

We use the following lemma for extending $\varphi$ to a 
proper $d$-edge precoloring $\varphi'$ of $Q_d$, such that an
edge $e$ of $Q_d$ is colored under $\varphi'$ if and only if
$e$ is precolored under $\varphi$ or $e$ is a conflict edge of
$h'$ with $L$.

	\begin{lemma}
	\label{gamma}
        Let $\alpha', \epsilon_0, \gamma, \kappa$ be constants such that
         $\alpha'=\max(\alpha + \gamma, \alpha +\epsilon_0)$,  
         $\kappa \geq \max(\alpha + \gamma, 
				\alpha +\epsilon_0, \gamma +\epsilon_0)$
         and
         $$d - \beta d -2 \alpha d -  2 \gamma d - 
				\dfrac{4\gamma}{\epsilon_0}d 
				- \frac{\alpha}{\epsilon_0}d \geq 1.$$  
        There is a proper $d$-edge precoloring $\varphi'$ of
				$Q_d$ satisfying the following:
        \begin{itemize}
        \item[(a)] 
				$\varphi'(uv)=\varphi(uv)$ for any edge  $uv$ of $Q_d$
				that is precolored under $\varphi$.
        \item[(b)] For every conflict edge $uv$ of $h'$, 
				$uv$ is colored under $\varphi'$ and
				$\varphi'(uv) \notin L(uv)$.
	
				\item[(c)] There are at most $\alpha' d$ precolored edges at 
				each vertex of $Q_d$ under $\varphi'$.
				
	\item[(d)] For every $12$-neighborhood $W$ in $Q_d$,
	there are at most $\alpha' d$ precolored edges with color
	 $i$ in $W$, $i =1,\dots, d$, under $\varphi'$.
	
	\item[(e)] For every $12$-neighborhood $W$ in $Q_d$,
	and every dimensional matching $M$,
	at most $\alpha' d$ edges of $M$ are precolored under $\varphi'$
	in $W$.
	\end{itemize}
	Furthermore, the edge coloring $h'$ of $Q_d$ 
  and the precoloring $\varphi'$ of $Q_d$ satisfy that 
	\begin{itemize}
	\item[(f)] For every $11$-neighborhood $W$ in $Q_d$, 
	and every dimensional matching $M$, at most
	$\kappa d$ edges of $M \cap E(W)$ are requested.
	\item[(g)] No vertex $x$ in $Q_d$ satisfies that $E_x$ contains more than 
	$\kappa d$ requested edges.
	\end{itemize}
	\end{lemma}
	
	\begin{proof}
	Consider the edge coloring $h'$ and the precoloring $\varphi$;
	for each $26$-neighborhood $W$,
	no dimensional matching in $Q_d$ contains 
	more than $\gamma d$ requested edges that are in $W$, 
	so the total number of requested edges in $W$ is not greater than
	$\gamma d^2$. Similarly, the total number of conflict edges in 
	each $27$-neighborhood $W$ is 
	not greater than $\gamma d^2$. 
	
	We shall construct the coloring $\varphi'$ by assigning a color
	to every conflict edge; this is done by iteratively constructing
	a $d$-edge precoloring $\phi$ of the conflict edges of $Q_d$;
	in each step we color a hitherto uncolored conflict edge, thereby
	transforming a conflict edge to a prescribed edge.
	At each step of transforming a conflict edge $uv$ into prescribed edge,
	the number of requested edges will increase by
	$2$. Hence, after constructing the proper $d$-edge precoloring 
	$\varphi'$, the total number of requested edges of each 
	$26$-neighborhood 
	is at most $\gamma d^2 + 2 \gamma d^2 = 3\gamma d^2$.
	
	Suppose now that we have constructed the precoloring $\varphi'$.
	A vertex $u$ in $Q_d$ is {\em $\varphi'$-overloaded} if $E_u$
	contains at least $\epsilon_0 d$ requested edges;
	note that no more than $\dfrac{3\gamma d^2}{\epsilon_0 d} 
	= \dfrac{3\gamma}{\epsilon_0}d$ 
	vertices of each $25$-neighborhood are $\varphi'$-overloaded.
	
	A color $c$ is {\em $\varphi'$-overloaded} in a $t$-neighborhood 
	$W$ if $c$ appears
	on at least $\epsilon_0 d$ edges in $W$ under $\varphi'$; note 
	that at most 
	$$\frac{\gamma d^2}{\epsilon_0d} + \frac{\alpha d^2}{\epsilon_0d}=
	\frac{\gamma + \alpha}{\epsilon_0}d$$ colors are $\varphi'$-overloaded in 
	each $25$-neighborhood $W$. These upper bounds hold
	for any choice of the precoloring $\varphi'$ obtained from $\varphi$
	by coloring the conflict edges of $Q_d$.

	Let $G$ be the subgraph of the hypercube $Q_d$
	induced by all conflict edges of $Q_d$.
	Let us now construct the $d$-edge coloring $\phi$ of $G$.
	We color the edges of $G$ by steps, and in each step we define a 
	list $\mathcal{L}(e)$ of allowed colors for 
	a hitherto uncolored edge $e =uv$ of $G$ by for 
	every color $c \in \{1,\dots, d\}$ 
	including $c$ in $\mathcal{L}(e)$ if 
	\begin{itemize}
	\item $c \notin L(uv)$,
	\item $c$ does not appear in $\varphi(u)$ or $\varphi(v)$,
	or on any previously colored edge of $G$ that is adjacent to $e$.
	\item $c$ is distinct from the color of the edge 
	$uu'$ (or $vv'$) under $h'$
	if $u'$ 
	(or $v'$) is $\varphi'$-overloaded.
	\item $c$ is not $\varphi'$-overloaded in the $25$-neighborhood
	of $e$.
	\end{itemize}
	Our goal is then to pick a color $\phi(e)$ from $\mathcal{L}(e)$ for $e$. 
	Given that this is possible for each edge of $G$, 
	this procedure clearly produces a $d$-edge-coloring $\phi$ of $G$, so that
	$\phi$ and $\varphi$ taken together form a proper $d$-edge precoloring
	of $Q_d$.
	
	Using the estimates above and the facts that $G$ has 
	maximum degree $\gamma d$, and
	$|\varphi(v)| \leq \alpha d$ for any vertex $v$ of $Q_d$,
	we have
	$$\mathcal{L}(e) \geq d - \beta d -2 \alpha d -  2 \gamma d
	- \dfrac{3\gamma}{\epsilon_0}d - \frac{\gamma + \alpha}{\epsilon_0}d,$$
	for every edge $e$ of $G$ in the process of constructing $\phi$, and 
	by assumption
	%
	%
	  $	\mathcal{L}(e) 
				\geq 1$.
	Thus, we conclude that we can choose an allowed color for each conflict edge
	so that the coloring $\phi$ satisfies the above conditions.
	This implies that taking $\phi$ and $\varphi$ together we obtain
	a proper $d$-precoloring $\varphi'$ of the edges of $Q_d$.
	Let us now prove that the precoloring $\varphi'$ satisfy
	the conditions in the lemma.
	
	Let $\alpha'=\max(\alpha + \gamma, \alpha +\epsilon_0)$.
	Then $\varphi'$ satisfies the following:
	\begin{itemize}
	\item 
				If $uv$ is precolored under $\varphi$,
        then $\varphi'(uv)=\varphi(uv)$. For every conflict edge $uv$, 
				there is a precolor $\varphi'(uv)$ such 
        that $\varphi'(uv) \notin L(uv)$.
	\item There are at most $\alpha' d$ precolored edges at each vertex.
	\end{itemize}

	Let us next prove that the precoloring $\varphi'$ satisfies conditions
	(d) and (e) of the lemma.
	Suppose that some $12$-neighborhood $W$ in $Q_d$ contains more than
	$\alpha' d$ precolored edges with color $i$, for some $i \in \{1,\dots,d\}$.
	Consider an edge $e$ in $W$ with $\phi(e)=i$. By the construction of $\phi$,
	in the $25$-neighborhood $W'$ of $e$ no color is $\varphi'$-overloaded.
	Note further that every $12$-neighborhood in $Q_d$ that
	$e$ lies in is contained in $W'$; thus $W$ is contained in $W'$, so
	the color $i$ is $\varphi'$-overloaded in $W'$, a contradiction. We conclude
	that condition (d) holds.
	A similar argument shows that condition (e) holds as well.
	
	Let us now turn to conditions (f) and (g).
	There are at most $\alpha' d$ precolored edges with color
	$i$, $i =1,\dots, d$, in every $12$-neighborhood in $Q_d$,
	and all edges that are in the same 
	dimensional mathing
	have the same color under $h'$. This implies that for each 
	$11$-neighborhood $W$
	and every dimensional matching $M$, the maximum number of requested edges 
	in $M$ that are in $W$ is $\alpha' d$. 
	Since $$\kappa \geq \max \{\alpha + \gamma, \alpha +\epsilon_0, 
	\gamma +\epsilon_0\},$$ condition (f) holds. Similarly, 
	at each step of transforming a conflict edge into a prescribed 
	edge under $\phi$, 
	we create $2$ new
	requested edges, $1$ at each vertex which is incident 
	with the conflict edge.
	Since the maximum degree in $G$ is $\gamma d$, and no vertex
	is $\varphi'$-overloaded,
	no vertex $x$ in
	$Q_d$ satisfies that $E_x$ contains more than 
	$\epsilon_0 d + \gamma d$ requested edges.
	Thus every 
	vertex $x$ in $Q_d$ satisfies that $E_x$ 
	contains at most $\kappa d$ requested edges.
	\end{proof}

\bigskip

\noindent
{\bf Step III:}
Let $\varphi'$ be the proper $d$-precoloring of $Q_d$
obtained in the previous step and $h'$ the $d$-edge coloring
of $Q_d$ obtained in Step I.
By a {\em clash edge (of $h'$)} in $Q_d$ we mean an edge which is both
prescribed and requested (under $\varphi'$).
We use the following lemma for constructing, from $h'$, a proper
$d$-edge coloring $h''$ of $Q_d$ with no clash edge.
The coloring $h''$ will also have the property that
every requested edge $e$ of $h''$
is adjacent to at most one prescribed edge $e'$ such that
$h''(e) = \varphi'(e')$.
	
	\vspace{0.5cm}
	
	\begin{center}
	
	\end{center}
	\begin{lemma}
	\label{request}
	Let $\kappa, \epsilon, \mu, \tau,
	\alpha'= \max(\alpha + \gamma,\alpha + \epsilon_0)$
	be constants such that 
	$\mu = 3\kappa + \epsilon +1$ and
	$$d-\tau d - 9\kappa d - 3\alpha' d - 3\epsilon d - 
	\dfrac{12\kappa}{\epsilon}d - 3 >0.$$
	By performing a sequence of 
	swaps on disjoint allowed $2$-colored $4$-cycles in $h'$, we
	obtain a proper $d$-edge coloring 
	$h''$ of $Q_d$  satisfying the following:
	\begin{itemize}
	 
	\item[(a)] There is no clash edge in $h''$.
	
	\item[(b)] For each requested edge $e$ of $h''$, $e$
	is adjacent to at most one edge $e'$ satisfying that 
	$h''(e) = \varphi'(e')$.
	
	\item[(c)] For each vertex $u \in V(Q_d)$,
	at most $2\kappa d+\epsilon d +1$ edges incident with
	$u$ appears in swaps for constructing $h''$ from $h'$.
	
	\item[(d)] For every $3$-neighborhood $W$ of $Q_d$,
	and every dimensional matching $M$,
	at most $2\kappa d+\epsilon d +1$ edges
	of $E(W) \cap M$
	appears in swaps for constructing $h''$ from $h'$.
	
	\item[(e)] For every $3$-neighborhood $W$ in $Q_d$, 
	and every dimensional matching $M$,
	there are at most $\mu d$ requested edges in $M \cap E(W)$.
	
	\item[(f)] No vertex in $Q_d$ is incident with 
	more than $\mu d$ requested edges.

	\end{itemize}
	
	\end{lemma}

	\begin{proof}
	An {\em unexpected edge of $h'$} is a clash edge or a
	requested edge $e$
	of $h'$ that is adjacent to more than one edge $e'$
	satisfying that $h'(e)= \varphi'(e')$.
	For constructing $h''$ from $h'$,
	we will perform a number of swaps on $2$-colored $4$-cycles, and
	we shall refer to this procedure as {\em $S$-swap}.
	In more detail, we are going to construct a set $S$ of disjoint
	allowed $4$-cycles, each such cycle containing exactly one
	unexpected edge in $h'$.
	An edge that belongs to a cycle in $S$ is called
	{\em used} in $S$-swap.
	
	Let us first deduce some properties that our set $S$,
	which is yet to be constructed, will satisfy.
	
	By Lemma \ref{gamma}, for every $11$-neighborhood $W$ in $Q_d$,
	and every dimensional matching $M$,
	the number of unexpected edges in $E(W) \cap M$
	is not greater than
	$\kappa d$.
	Suppose we have included a $4$-cycle $C$ in
	$S$. Every edge in $C$ is at distance at
	most $1$ from the unexpected edge contained in $C$; this implies that for 
	every $10$-neighborhood $W$ in $Q_d$, the total number
	of edges in $W$ that are used in $S$-swap is at most $4\kappa d^2$.
	
	A vertex $u$ in $Q_d$ is {\em $S$-overloaded} if $E_u$ contains
	at least $\epsilon d$
	edges that are used in $S$-swap; note that no more than
	$\dfrac{4\kappa d^2}{\epsilon d}=\dfrac{4\kappa}{\epsilon}d$ vertices
	of each $9$-neighborhood are $S$-overloaded.
	A dimensional matching $M$ in $Q_d$ is {\em $S$-overloaded} 
	in a $t$-neighborhood $W$
	if $M \cap E(W)$ contains
	at least $\epsilon d$ edges that are used in $S$-swap; 
	note that for each $10$-neighborhood $W$, no more than 
	$\dfrac{4\kappa}{\epsilon}d$ dimensional matchings
	of $Q_d$ are $S$-overloaded in $W$.
		 
	Using these facts, let us now construct our set $S$ by steps; at each step
	we consider an unexpected edge $e$ 
	and include an allowed $2$-colored $4$-cycle containing $e$ in $S$.
	Initially, the set $S$ is empty. Next, 
	for each unexpected edge $e=uv$ in $Q_d$, 
	there are at least $d- \tau d$ allowed cycles 
	containing $e$.
	We choose an allowed cycle $uvztu$
	which contains $e$ and satisfies the following:

	\begin{itemize}
	\item[(1)] $z$ and $t$ and the dimensional matching that 
	contains $vz$ and $ut$ 
	are not $S$-overloaded in the $9$-neighborhood $W_e$
	of $e$; this eliminates at most $\dfrac{12\kappa}{\epsilon}d$ choices.
	
	Note that with this strategy
	for including $4$-cycles in $S$, after completing the
	construction of $S$,
	every vertex is incident with at most
	$2\kappa d+\epsilon d +1$ edges that are used in $S$-swap;
	that is, condition (c) holds.
	
	Furthermore, after we have constructed the set $S$,
	no dimensional matching is $S$-overloaded
	in a $3$-neighborhood of $Q_d$; this
	follows from the fact that
	every $3$-neighborhood $W'$ in
	$Q_d$ that $ut$, $vz$ or $zt$
	belongs to is contained in $W_e$. Moreover, this implies
	that condition (d) holds.
	
	
	
	\item[(2)] None of the edges $vz$, $zt$, $ut$ are 
	prescribed, or requested, or used before
	in $S$-swap. 
	
	All possible choices for these edges
	are in the $3$-neighborhood $W_e$ of $e$ in $Q_d$.
	By Lemma \ref{gamma}, no vertex  in $W_e$, or 
	subset of a dimensional matching that is in $W_e$,
	contains more than $\kappa d$
	requested edges or $\alpha' d$ prescribed edges. Moreover, $S$-swap
	uses at most $2\kappa d+\epsilon d +1$ edges
	at each vertex and in each subset of a dimensional matching
	contained in $W_e$. Hence, these restrictions
	eliminate at most 
	$3 (\kappa d+\alpha' d) + 3(2 \kappa d+\epsilon d +1)$ 
	or $9 \kappa d+3 \alpha'd+3\epsilon d+3$  choices.
	\end{itemize}
	It follows that we have at least
	$$d -\tau d - 9\kappa d - 3\alpha' d - 3\epsilon d - 
	\dfrac{12\kappa}{\epsilon}d - 3$$
	choices for an allowed cycle $uvztu$ which contains $uv$.
	By assumption, this expression is greater than zero,
	so we conclude that there is a cycle satisfying these conditions,
	and thus we may construct the set $S$ by iteratively adding disjoint allowed
	$2$-colored $4$-cycles such that each cycle contains a unexpected edge.
	%
	
	After this process terminates we have a set $S$ of disjoint 
	allowed cycles; 
	we swap on all
	the cycles in $S$ to obtain the coloring $h''$.
	Note that for the cycle $uvztu$ constructed above, 
	since none of the edges $vz$, $zt$, $ut$
	are prescribed or requested,
	$\{\varphi'(z) \cup \varphi'(t)\}$ does not contain the color $h'(uv)$;
	so after swapping colors
	on the cycle $uvztu$, none of the edges edges $uv, vz$, $zt$, $ut$ are
	unexpected edges in the obtained coloring; that is, condition
	(a) and (b) hold.
	
	Let us finally verify that conditions (e) and (f) hold.
	As noted above, 
	for every dimensional matching $M$ and every $3$-neighborhood $W$,
	$S$-swap uses at most $2\kappa d + \epsilon d + 1$ from $E(W) \cap M$.
	Moreover, by Lemma \ref{gamma}, $E(W) \cap M$ contains at most
	$\kappa d$ requested edges under $h'$ with respect to $\varphi'$.
	Thus the proper coloring $h''$ satisfies that for
	every dimensional matching $M$ and
	for every $3$-neighborhood $W$ in $Q_d$,
	at most $\mu d$ requested edges are contained in $E(W) \cap M$. 
	Similarly, no vertex $x$ in $Q_d$ satisfies that
	$E_x$ contains more than $\mu d$ requested edges.
	\end{proof}

\bigskip

\noindent
{\bf Step IV:} Let $h''$ be the proper $d$-edge coloring of $Q_d$
obtained in the previous step and let $\varphi'$ be the
precoloring of $Q_d$ obtained in Step II. Then $h''$ and $\varphi'$
satisfies (a)-(f) of Lemma \ref{request}, and also the following:
\begin{itemize}
	\item each vertex of $Q_d$ is incident with at most $\alpha' d$
	edges that are precolored under $\varphi'$;
	\item for every $12$-neighborhood $W$ and every 
	dimensional matching $M$ in $Q_d$, at most $\alpha' d$
	edges of $M$ are precolored under $\varphi'$ in $W$;

	\item for every $12$-neighborhood $W$, 
	there are at most $\alpha' d$ edges that are precolored $i$
	under $\varphi'$ in $W$, $i=1,\dots,d$.
\end{itemize}
	
	As in the proof of Lemma \ref{request},
	we say that an edge $e$ in $Q_d$ with $h'(e) \neq h''(e)$
	{\em is used in $S$-swap}.
	Note that since in every $12$-neighborhood, the number
	of edges that are precolored $i$ ($i \in \{1,\dots,d\}$)
	is at most $\alpha' d$,
	and since the number of precolored edges in a
	subset of a dimensional matching of $Q_d$ that is contained
	in a $12$-neighborhood
	is also bounded,
	there is a bounded number of edges colored $i$ under $h''$ that
	have been used in $S$-swap in each $3$-neighborhood.
	Moreover, since $S$-swap uses a bounded
	number of edges at each vertex, and in the intersection
	of every dimensional matching and $3$-neighborhood
	(by condition (c) and (d) in Lemma \ref{request}),
	most edges in $Q_d$ are in a large number of allowed $2$-colored
	$4$-cycles under $h''$. Those two properties are central for completing
	the proof of Theorem \ref{th:main2} in Step IV; this is done by proving the following lemma.
	
	
	\begin{lemma}
	\label{complete}
	Let $\kappa, \epsilon, \tau, \mu= 3\kappa + \epsilon +1, 
	\alpha'=(\alpha + \gamma,\alpha + \epsilon_0)$ be constants such that 
	$$d - 64\mu d - 64\alpha' d-  32\kappa d-  32\epsilon d - 10\beta d -3\tau d - 
	\dfrac{266\alpha'}{\epsilon}d- 86 >0.$$
	There is a proper $d$-edge coloring of $Q_d$
	that is an extension of $\varphi'$ and which avoids $L$.
	\end{lemma}
	
	\begin{proof}
	If $h''(e)=\varphi'(e)$ for all precolored edges $e$ then we do nothing;
	$h''$ is the required proper edge coloring. 
	\begin{figure}
	\begin{center}
	\hspace{0.045cm}\includegraphics[scale=0.2]{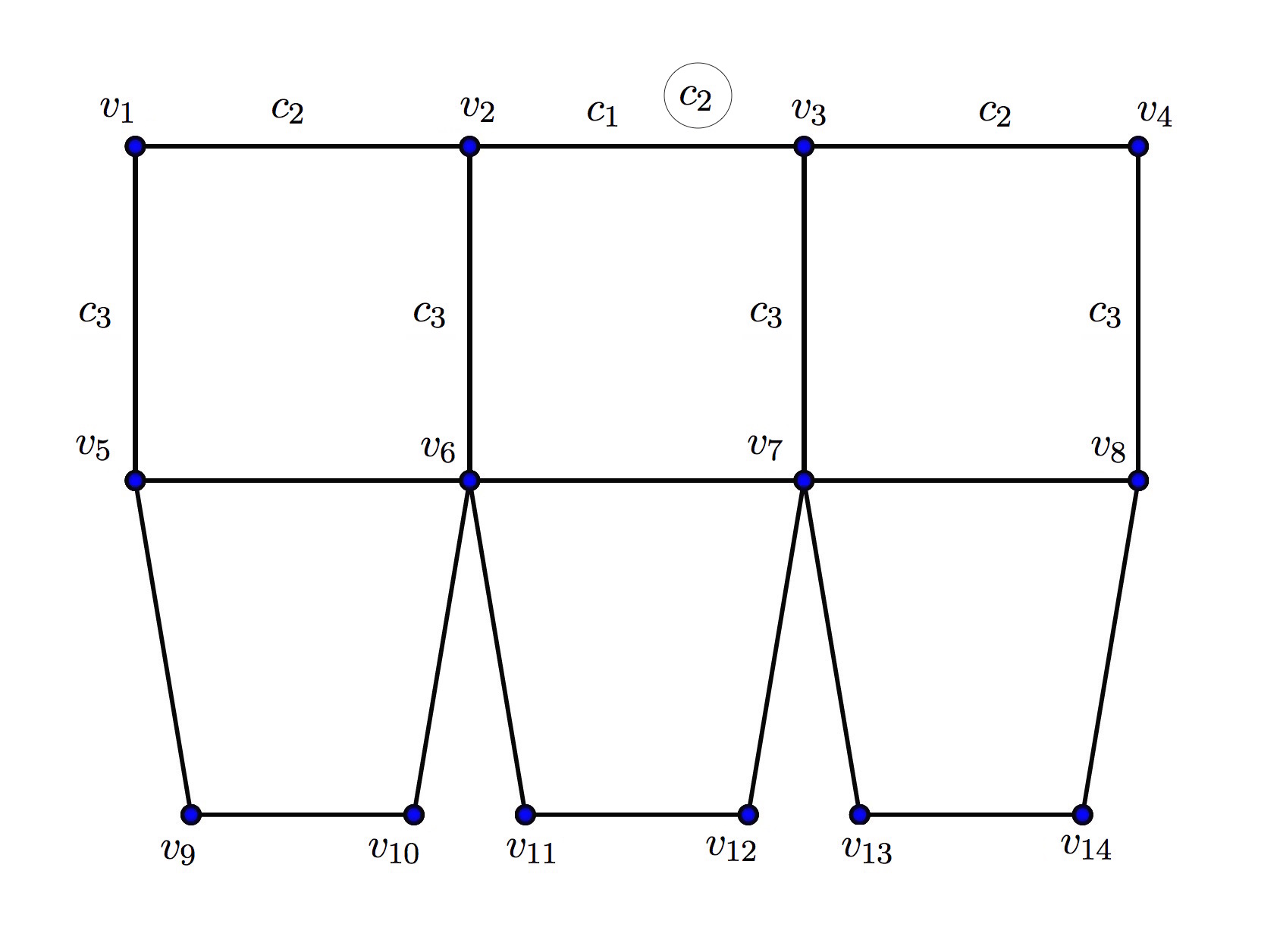}\hspace{15px}
	\caption{An example of a configuration $T_e$.}
	\label{fig:sw}
	\end{center}
	\end{figure}
	Else, we construct a set $T \subseteq E(Q_d)$,
	such that performing a sequence of swaps on allowed $2$-colored
	$4$-cycles of the subgraph of $Q_d$
	induced by $T$, we obtain the required extension of $\varphi'$.
	We refer to this construction as {\em $T$-swap}.
	For each $\varphi'$-precolored edge $e$, the set $T$ will contain
	a subset $T_e$ of edges associated with $e$; if $e$ and $e'$
	are distinct $\varphi'$-precolored edges of $Q_d$, then
	we will have $T_e \cap T_{e'} = \emptyset$.
	An example of a subset $T_e$ can be seen in Figure 		
	\ref{fig:sw}, where $v_2v_3$ is a prescribed edge,
	and $v_1v_2$ and $v_3v_4$ are requested.
	Since distinct sets $T_e$ and $T_{e'}$ are disjoint,
	every requested edge is in at most one set $T_e$;
	this property is ensured by 
	Lemma \ref{request} (b).
	
	An edge that belongs to $T$ is called
	{\em used} in $T$-swap.
	By Lemma \ref{gamma}, for every $12$-neighborhood $W$ in $Q_d$,
	and every dimensional matching $M$,
	at most $\alpha' d$ edges of $M \cap E(W)$ are precolored under $\varphi'$. 
	For each configuration $T_e$ in $T$, every edge of $T_e$
	is at distance at most $2$ from the prescribed edge; this implies 
	that for every $10$-neighborhood $W$ in $Q_d$, the total number
	of edges in $W$ that are used in $T$-swap is at most $19\alpha' d^2$.
	
	A vertex $u$ in $Q_d$ is {\em $T$-overloaded} 
	if at least $\epsilon d$ edges from $E_u$ are used in $T$-swap; 
	note that no more than 
	$\dfrac{19\alpha' d^2}{\epsilon d}=\dfrac{19\alpha'}{\epsilon}d$ vertices
	of each $9$-neighborhood are $T$-overloaded.
	A dimensional matching $M$ in $Q_d$ is {\em $T$-overloaded}
	in a $t$-neighborhood $W$ if $M \cap E(W)$ contains 
	at least $\epsilon d$ edges that are used in $T$-swap;
	note that 
	for each $10$-neighborhood $W$
	no more than
	$\dfrac{19\alpha'}{\epsilon}d$ dimensional matchings
	 are $T$-overloaded in $W$.
	
	Consider the setup in Figure 1. We now describe how to
	construct the set $T_e$ for the prescribed edge $e=v_2v_3$.
	Suppose that
	${\varphi'}(v_2v_3) = c_2 \neq h''(v_2v_3) =c_1.$ 
	Since every vertex in $Q_d$ has degree $d$ we initially have
	at least $d-3$ choices for a subgraph as in Figure 1.


	Let $v_1v_2$ and $v_3v_4$ be the 
	edges adjacent to $v_2v_3$ that are colored $c_2$.
	The set $T_{v_2v_3}$ will consist 
	of edges incident
	with $14$ vertices $v_1, \dots, v_{14}$.
	We shall choose the vertices $v_5, \dots, v_{14}$ such that they satisfy
	a number of properties:
	
	\begin{itemize} 
	\item[(1)] $v_5,\dots, v_{14}$ and the dimensional matchings
	that contain $v_1v_5$, $v_2v_6$, $v_3v_7$, $v_4v_8$
	$v_5v_9$, $v_6v_{10}$, $v_6v_{11}$,
	$v_7v_{12}$, $v_7v_{13}$, $v_8v_{14}$
	are not $T$-overloaded in the $9$-neighborhood $W_e$
	of $e$. 
	
	These edges are in at most four dimensional 
	matchings and we select 10 new vertices, so this eliminates at most 
	$\dfrac{19\times 14\alpha'}{\epsilon}d$ or
	$\dfrac{266\alpha'}{\epsilon}d$ choices.

	Moreover, every 
	$2$-neighborhood $W$ in $Q_d$ that one of these edges
	lie in is contained in $W_e$, so with this strategy and
	with the bounds on the number of requested and
	prescribed edges under $h''$, in the process of choosing the set $T$,
	for every dimensional matching $M$ and every
	$2$-neighborhood $W$ in $Q_d$, $E(W) \cap M$
	contains at most $$3\mu d+3\alpha' d+ (\epsilon d -1) + 4=
	3\mu d+3\alpha' d+ \epsilon d +3$$
	edges that are used in $T$-swap. Similarly, every vertex
	is incident with at most 
	$$3\mu d+3\alpha' d+ (\epsilon d -1) + 5=3\mu d+3\alpha' d+ 
	\epsilon d +4$$ edges that
	are used in $T$-swap; these upper bounds follow from the facts that
	the maximum number of edges of $T_e$ incident with one vertex is $5$, 
	and the maximum
	number of edges from a given dimensional matching in $T_e$ is $4$.
	
	\item[(2)] None of the edges $v_1v_5$, $v_2v_6$,
	$v_3v_7$, $v_4v_8$ are prescribed or requested or used before
	in $T$-swap. 
	
	Since all the possible choices for these edges
	are in the $2$-neighborhood $W_e$ of $e$ in $Q_d$, and under $h''$ no vertex
	contains more than $\alpha' d$ prescribed edges and $\mu d$ requested edges,
	$T$-swap uses at most $3\mu d+3\alpha' d+ \epsilon d +4$
	edges at each vertex, this condition eliminates at most 
	$$4 \times (\mu d+\alpha' d)+4 \times (3\mu d+3\alpha' d+ \epsilon d + 4) =
	16\mu d + 16\alpha' d+ 4\epsilon d + 16$$ choices.
	
	\item[(3)] $v_1v_5$, $v_2v_6$, $v_3v_7$, $v_4v_8$ are not
	used before in $S$-swap. 
	
	Note that this condition ensures that
	$h''(v_1v_5) = h''(v_2v_6) = h''(v_3v_7) = h''(v_4v_8)=c_3$, where $c_3$ 
	is the most common color of the dimensional matching that contains $v_1v_5$. 
	This eliminates at most $8\kappa d + 4\epsilon d + 4$ choices based on 
	the conditions (c) and (d) of Lemma \ref{request}.
		
	\item[(4)] The three cycles $v_1v_2v_6v_5v_1$, $v_2v_3v_7v_6v_2$, 
	$v_3v_4v_8v_7v_3$ are allowed 
	before $S$-swap. 
	
	Each of the edges $v_1v_2$, $v_2v_3$, $v_3v_4$ belongs to at 
	most $\tau d$ non-allowed cycles, so this eliminates at most $3 \tau d$ choices.
	
	\item[(5)] $c_2 \notin \{L(v_1v_5) \cup L(v_4v_8)\}$. 
	
	Color $c_2$ appears at most
	$\beta d$ times in lists of the set of edges incident to $v_1$ (or $v_4$), so
	this eliminates at most $2\beta d$ choices.
	
	\item[(6)] $c_1 \notin \{L(v_2v_6) \cup L(v_3v_7)\}$. 
	
	Similarly, this eliminates
	at most $2\beta d$ choices.
	
	\item[(7)] If $v_1 v_2$ is in a dimensional matching of color 
	$c_2$ of $h''$, then we choose 
	$v_5, v_6$ such that $v_5v_6$ is not used in $S$-swap
	or $T$-swap, and $v_5v_6$ is not prescribed or requested.
	Note that the conditions imply that $h''(v_5v_6) =c_2$; in this case we 
	choose $v_9, v_{10}$ arbitrarily.
	Based on the restriction of edges used in each dimensional matching
	and $v_5v_6 \neq v_1v_2$, these restrictions eliminate at most
	$$(2\kappa d+\epsilon d +1)+(3\mu d+3\alpha' d+ \epsilon d + 4)+1
	=2\kappa d+ 3\alpha' d + 2\epsilon d + 3\mu d +6$$ choices.

	Else, we choose $v_5, v_6, v_9, v_{10}$ such that
	\begin{itemize}
	
	\item[(a)] $v_5v_6$ is not used in $S$-swap
	(and also $v_5v_6 \neq v_1v_2$);
	this eliminates at most $2\kappa d+\epsilon d +1$ choices. We can assume
	that $v_5v_6$ is in a dimensional matching of color $c_4$.
		
	\item[(b)] $v_5v_6$ is in an allowed cycle
	$v_5v_6v_{10}v_9v_5$ with color $c_2$ (which means
	$h''(v_5v_9)=h''(v_6v_{10})=c_2$, 
	$h''(v_9v_{10})=h''(v_5v_6)=c_4$), $c_4 \notin \{L(v_5v_9) \cup L(v_6v_{10})\}$
	and $c_2 \notin \{L(v_5v_6) \cup L(v_9v_{10})\}$. 
	
	In the $3$-neighborhood $W$ of $e$ in $Q_d$,
	$S$-swap uses at most $2\kappa d+\epsilon d +1$ edges in the dimensional
	matching of color $c_2$, so this eliminates at most 
	$4\kappa d + 2\epsilon d + 2$ choices for $v_5v_9$
	and $v_6v_{10}$. We also require that $v_9v_{10} \neq v_5v_6$
	and that $v_9v_{10}$ is not used in $S$-swap to make sure
	$h''(v_9v_{10})=h''(v_5v_6)=c_4$; this eliminates at most 
	$2\kappa d + \epsilon d + 2$ choices.
	
	Since $c_2$ occurs $\beta d$ times in the 
	subset of the dimensional matching of color $c_4$ contained
	in the $27$-neighborhood $W$ of $e$ in $Q_d$,
	and $c_4$ occurs $\beta d$ times in the subset of the
	dimensional matching of color $c_2$ contained in $W$, 
	the two conditions $c_4 \notin \{L(v_5v_9) \cup L(v_6v_{10})\}$
	and $c_2 \notin \{L(v_5v_6) \cup L(v_9v_{10})\}$ 
	eliminate at most $2\beta d$ choices.
	
	\item[(c)] $v_5v_6$, $v_5v_9$, $v_9v_{10}$,
	$v_6v_{10}$ are not prescribed or requested or used before
	in $T$-swap. 
	
	Since all the possible choices for these edges
	are in the $2$-neighborhood $W$ of $e$ in $Q_d$, 
	this eliminates
	at most $4 \times (\mu d + \alpha' d) + 
	4 \times (3\mu d+3\alpha' d+ \epsilon d + 4)$ 
	choices.
	
	\end{itemize}
	So in both cases, the choosing process eliminates at most
	$$16\mu d + 16\alpha' d +8\kappa d+ 8\epsilon d + 2\beta d + 21$$ choices.
	
	\item[(8)] $v_7, v_8, v_{14}, v_{13}$ is chosen with same strategy as 
	$v_5, v_6, v_9, v_{10}$. 
	
	Similarly,
	this eliminates at most $16\mu d + 16\alpha' d +8\kappa d+ 
	8\epsilon d + 2\beta d + 21$ choices.
	 
	 \item[(9)] $v_6, v_7, v_{11}, v_{12}$ is chosen with same strategy with 
	$v_5, v_6, v_9, v_{10}$ but the color $c_2$ is replaced by $c_1$. 
	
	Again, 
	this eliminates at most 
	$16\mu d + 16\alpha' d +8\kappa d+ 8\epsilon d + 2\beta d + 21$ choices.
	\end{itemize}
	
	Summing up, we conclude that in total, there are at most  
	$$64\mu d + 64\alpha' d+ 32\kappa d+ 32\epsilon d + 10\beta d +3\tau d
	+\dfrac{266\alpha'}{\epsilon}d+83$$ forbidden 
	choices for the configuration $T_e$.
	
	This implies that we have
	$$d-3- 64\mu d - 64\alpha' d-  32\kappa d- 32\epsilon d - 10\beta d -3\tau d- 
	\dfrac{266\alpha'}{\epsilon}d - 83$$
	or 
	$$Z=d - 64\mu d - 64\alpha' d-  32\kappa d-  32\epsilon d - 10\beta d 
	-3\tau d - \dfrac{266\alpha'}{\epsilon}d- 86$$
	choices for a configuration $T_{e'}$ in the process of constructing $T$,
	whenever $e'$ is a prescribed edge.
	
	By assumption,  $Z>0$, so there is a set $T_{v_2v_3}$ that
	satisfies all the above conditions. 
	We add this set to
	$T$ and apply this procedure for all prescribed
	edgess $uv$ with $h''(uv) \neq \varphi'(uv)$. 
	Since the resulting subsets of $T$ are disjoint, we 
	can do the following transformation for each subset $T_{v_2v_3}$
	as above.
	\begin{itemize}
	\item If $h''(v_5v_6) \neq c_2$, then interchange colors of the 
	cycle $v_5v_6v_{10}v_9v_5$.
	
	\item If $h''(v_6v_7) \neq c_1$,  then interchange colors 
	of the cycle $v_6v_7v_{12}v_{11}v_6$.
	
	\item If $h''(v_7v_8) \neq c_2$, then interchange colors of 
	the cycle $v_7v_8v_{14}v_{13}v_7$.
	
	\item Next, interchange colors of the cycles $v_1v_2v_6v_5v_1$ 
	and $v_3v_4v_8v_7v_3$.
	
	\item Finally, interchange colors of the cycle $v_2v_3v_7v_6v_2$. 
	
	\end{itemize}
	In the resulting edge coloring obtained from $h''$, 
	$v_2v_3$ is colored $c_2$.
	Moreover, it follows from conditions $(3)$, $(4)$, $(5)$, $(7)$ that
	we do not create any new conflict edges by performing these swaps.
	We thus conclude that
	by repeating this swapping procedure for every prescribed edge,
	we obtain a new proper $d$-edge coloring 
	which agrees with the precoloring $\varphi'$.
	%
	%
	\end{proof}

We have proved that it is possible to complete 
all the steps I-IV outlined in Section 2,
thereby obtaining an extension of $\varphi$ that avoids $L$; 
this completes the proof of Theorem \ref{th:main2}.
\end{proof}

Let us now turn to the proof of Theorem \ref{secondth}.
As we shall see, Property \ref{prop:stand}
of the standard edge coloring $h$ of $Q_d$ trivially yields
the result.

\begin{proof}[Proof of Theorem \ref{secondth} (sketch).] 
	Let $\varphi$
	be a $d$-edge precoloring of $Q_d$, 
	and $L$ a $\beta$-sparse list assignment for
	the non-precolored edges of $Q_d$, such that any
	edge $e$ which is either
	precolored or satisfies $L(e) \neq \emptyset$ belongs
	to a distance-3 matching $M$ 
	in $Q_d$. Let $h$ be the standard $d$-edge coloring of $Q_d$ defined above.
	
	Now, by arbitrarily
	picking a color from the set $\{1,\dots, d\} \setminus L(e')$ 
	for each conflict edge $e'$,
	we can construct 
	a precoloring $\varphi'$ from $\varphi$ such that an edge $e$ of $Q_d$
	is precolored under $\varphi'$ if and only if $e$ is precolored
	under $\varphi$ or $e$ is a conflict of $h$ with $L$.
	Furthermore, any $2$-colored $4$-cycle $C$
	with colors $c_1$ and $c_2$ under $h$, and satisfying that there is an edge
	$e \in E(C)$ with $\varphi'(e) =c_1$ and $h(e)= c_2$
	is allowed.
	Moreover, since edges that are precolored under $\varphi'$ are at distance
	at least $3$ from each other, two $4$-cycles containing distinct precolored
	edges are disjoint. Now, by Property \ref{prop:stand}, every
	edge in $Q_d$ is contained in $d-1$ $2$-colored $4$-cycles under $h$;
	thus, we may complete the proof by simply swapping on
	a suitable set of disjoint $2$-colored $4$-cycles.	
\end{proof}



\section{Upper bounds and further problems}
We have proved that there are constants $\alpha$ and $\beta$ such  that every $\alpha$-dense $d$-edge precoloring of $Q_d$ 
can be extended to a proper $d$-edge coloring avoiding any given $\beta$-sparse list assignment for $Q_d$.  The values we have found for 
$\alpha$ and $\beta$ are quite small, to a large extent due to the
calculations in
Lemma \ref{alpha}.

Let us briefly compare our results obtained in this paper with corresponding
results for complete bipartite graphs. Recall that a list assignment $L$
for $K_{n,n}$ is {\em $\beta$-sparse} if
each edge $e$ of $K_{n,n}$ is assigned a list $L(e)$
of at most $\beta n$ forbidden colors from $\{1,\dots,n\}$, 
and at every vertex $v$ each color appears in lists of
at most $\beta n$ edges adjacent to $v$; similarly an $n$-edge precoloring
of $K_{n,n}$ is {\em $\alpha$-dense} if
every color is used at most $\alpha n$ times in the precoloring
and at every vertex $v$ at most $\alpha n$ edges incident to $v$
are precolored.
For $K_{n,n}$ Daykin and H\"aggkvist \cite{DH} conjectured that $\alpha=1/4$ 
is the optimal value, and H\"aggkvist conjectured that $\beta=1/3$ is optimal.  The currently best value is $\alpha=1/25$, as proven in \cite{BKLOT}. The best known 
value for $\beta$ is given in \cite{ACO} is far smaller, 
due to probabilistic tools. 
That one can simultaneously take $\alpha$ and $\beta$ 
to be positive was proven 
in \cite{ACM}.

For the hypercube $Q_d$, the following general proposition
yields an upper bound on the values of $\alpha$ and $\beta$
in Theorem \ref{mainth}.

\begin{proposition}
\label{prop:bound}
Let $G$ be a $d$-regular d-edge-colorable graph.

\begin{itemize}
\label{bound}
	\item[(i)]  If every $d$-edge precoloring of $G$, satisfying that
	each vertex of $G$ is incident to at most $\alpha d$ precolored edges,
	is extendable, then
	$\alpha < \frac{1}{2}$.
	
	\item[(ii)] If every list assignment $L$, such that $|L(e)| \leq \beta d$
for each edge $e \in E(G)$, and for each vertex $v$
each color appears in at most $\beta d$ lists of edges incident with $v$,
is avoidable, then
	$\beta < \frac{1}{2}$.
	
	\item[(iii)] If every precoloring as in (i)
	is extendable to a coloring avoiding any list assignment as in (ii), then
	$\alpha +\beta < \frac{1}{2}$.
\end{itemize}

\end{proposition}
\begin{proof}
	\begin{itemize}
		\item[(i)]
		Let $u_1u_2$ be an edge of $G$. We define an edge precoloring
		$\varphi$ of $G$ by coloring $\lceil d/2 \rceil$
		edges incident with $u_1$ and distinct from $u_1u_2$ by colors
		$1,\dots, \lceil d/2 \rceil$; next, color $\lceil d/2 \rceil$
		edges incident
		to $u_2$ and distinct from $u_1u_2$ by colors 
		$\lceil d/2 \rceil+1, \dots, d$.
		This yields an edge $d$-precoloring which is not extendable 
		to a proper $d$-edge coloring, 
		so necessarily $\alpha<1/2$.

		\item[(ii)]
		Let $u_1u_2$ be an edge of $G$.   Next, to $\lceil d/2 \rceil$ 
		edges incident with $u_1$, but not $u_2$,  assign identical 
		color lists containing all the 
		colors $1,\dots, \lceil d/2 \rceil$. Similarly assign 
		to $\lceil d/2 \rceil$ edges incident with $u_2$, but not $u_1$, 
		identical color lists containing 
		all the colors $\lceil d/2 \rceil+1,\dots, d$. Now, 
		since apart from $u_1u_2$, there are at 
		most $\lceil d/2 \rceil-1$ edges incident with $u_1$, where
		colors $1, \dots, \lceil d/2 \rceil$ are not forbidden, 
		we must have that $u_1u_2$ is colored with a color from 
		$1, \dots, \lceil d/2 \rceil$ in any proper 
		$d$-edge coloring of $G$ avoiding the list assignment; 
		similarly by the restrictions at $u_2$, $u_1u_2$ must be 
		colored with a color from 
		$\lceil d/2 \rceil+1, \dots, d$ in any 
		coloring of $G$ avoiding the list assignments at $u_2$. 
		This is clearly not possible, so the list assignment 
		is unavoidable, and thus $\beta < 1/2$.

		\item[(iii)]
		The precoloring and list-assignments defined 
		above can be combined in the following way 
		(we assume that $1/2 > \beta > \alpha$ and that $\alpha d$ and 
		$\beta d$ are integers):

		Let $u_1u_2$ be an edge as above, 
		let $H_1$ be the star induced by $u_1$ and its
		neighbors except for $u_2$, and $H_2$ the corresponding star for $u_2$.	
		
		We now consider the assignment where in $H_1$ 
		there are $\alpha d$ precolored edges incident with $u_1$ 
		using colors $d - \alpha d +1, \dots, d$; 
		moreover, exactly
		$\beta d$ edges in $H_1$ incident with $u_1$, 	
		distinct from the precolored ones, are assigned 
		identical lists with colors $1, \dots, \beta d$.
		Similarly, in $H_2$ there are $\alpha d$  
		precolored edges incident with $u_2$ using 
		colors $1, \dots, \alpha d$; moreover, 
		there are precisely $\beta d$ edges in 
		$H_2$ incident with $u_2$, distinct from the precolored edges, 
		all of which are assigned identical color 
		lists  containing colors $d- \beta d+1, \dots, d$.

		Now, for any proper $d$-edge coloring $f$ of $G$ which
		is an  extension of the precoloring and which avoids the 
		the list assignment, the colors $1, \dots, \beta d$ 
		must appear on edges incident with $u_1$ which are neither 
		precolored nor are assigned a non-empty list of forbidden colors. 
		By a similar argument for $H_2$, we must have that for any coloring $f$
		which is an extension of the precoloring and also avoids
		the list assignment, colors 
		$d - \beta d +1, \dots, d$	 must appear on edges incident with $u_2$ 
		which are neither precolored nor are assigned a non-empty list of 
		 forbidden colors. 
		Note 
		that both $u_1$ and $u_2$ are incident 
		with exactly $d- \beta d - \alpha d$ edges 
		which are neither precolored nor contain a non-empty 
		list of forbidden colors.
		Thus if $d- \alpha d - \beta d \leq \beta d$, 
		then the edge $u_1u_2$ must receive a color both 
		from the set $\{1, \dots, \beta d\}$ and from the set 
		$\{d - \beta d +1, \dots, d\}$ under $f$.  
		Moreover, if $\beta d \leq d - \beta d +1$, 
		then these sets are disjoint, implying that there 
		is no extension of the precoloring 
		which avoids the given list assignment. 
		Here, by choosing $\beta$ close to $1/2$ and $\alpha$ small, 
		we can make the sum $\alpha + \beta$ arbitrarily close to $1/2$. 

	\end{itemize}
\end{proof}

Returning to the setup of Theorem \ref{mainth}, 
we have attempted to find constructions which yield 
better upper bounds for $\alpha$ and $\beta$ for the hypercubes,
but have not been able to do so. 
Moreover, the conditions (ii) and (iii) for a precoloring of $Q_d$
to be $\alpha$-dense are not probably not best possible in terms
of size of the neighborhoods.
Those conditions are required in our proof, 
but might be far stronger 
than what is actually needed in order 
for a compatible edge coloring to exist.  
Nonetheless it would be interesting to see how 
far Theorem \ref{mainth} can be improved in its current form
(possibly with decreased size of the neighborhoods).
\begin{problem}
	What are the optimal values for $\alpha$ and $\beta$ in Theorem \ref{mainth}? 
\end{problem}

Our focus here has been the family of hypercubes but of course the type of problem we have considered is interesting for more general graphs as well.  
The examples in \cite{EGHKPS} show that in order to get results similar 
to those for $K_{n,n}$,  and those given in this paper, one must impose some structural conditions on the considered family of graphs.  
Both $K_{n,n,}$ and $Q_d$ are well connected bipartite graphs and it would be interesting to see how far Proposition \ref{bound} can be improved for this general class of graphs.
\begin{problem}
	Given a precoloring and a list assignment as in Proposition \ref{prop:bound},
	what are the optimal values for $\alpha$ and $\beta$ for the family of $d$-regular, $d$-edge connected, bipartite graphs? 
	
\end{problem}
Here the cases closest to our results are of course those where $d$ 
is a function of the number of vertices in the graph.

Finally, as mentioned in the introduction, our proof method easily give us Theorem \ref{secondth} where the edges 
which are precolored or have non-empty lists of forbidden colors
on them are forced to lie in a distance-3 matching.  
Here it is natural to ask if this result holds for distance-2 matchings 
 as well.

\bibliographystyle{amsalpha}
\bibliography{papers}

\newcommand{\etalchar}[1]{$^{#1}$}
\providecommand{\bysame}{\leavevmode\hbox to3em{\hrulefill}\thinspace}
\providecommand{\MR}{\relax\ifhmode\unskip\space\fi MR }
\providecommand{\MRhref}[2]{%
  \href{http://www.ams.org/mathscinet-getitem?mr=#1}{#2}
}
\providecommand{\href}[2]{#2}
\begin{thebibliography}{BKL{\etalchar{+}}16}

\bibitem[ACM16]{ACM}
L.~J. {Andr{\'e}n}, C.~J. {Casselgren}, and K.~{Markstr{\"o}m},
  \emph{{Restricted completion of sparse partial Latin squares}}, ArXiv
  e-prints (2016).

\bibitem[AC{\"O}13]{ACO}
Lina~J. Andr{\'e}n, Carl~Johan Casselgren, and Lars-Daniel {\"O}hman,
  \emph{Avoiding arrays of odd order by {L}atin squares}, Combin. Probab.
  Comput. \textbf{22} (2013), no.~2, 184--212. \MR{3021331}

\bibitem[AH83]{andersen1983thank}
LD~Andersen and AJW Hilton, \emph{Thank evans!}, Proceedings of the London
  Mathematical Society \textbf{3} (1983), no.~3, 507--522.

\bibitem[AM01]{AM2001}
Michael~O. Albertson and Emily~H. Moore, \emph{Extending graph colorings using
  no extra colors}, Discrete Math. \textbf{234} (2001), no.~1-3, 125--132.
  \MR{1826825}

\bibitem[And10]{andren2010latin}
Lina~J Andr{\'e}n, \emph{On latin squares and avoidable arrays}, Ph.D. thesis,
  Ume{\aa} universitet, Instituionen f{\"o}r matematik och matematisk
  statisitik, 2010.

\bibitem[Bar13]{B2013}
Padraic Bartlett, \emph{Completions of {$\epsilon$}-dense partial {L}atin
  squares}, J. Combin. Des. \textbf{21} (2013), no.~10, 447--463. \MR{3090723}

\bibitem[BKL{\etalchar{+}}16]{BKLOT}
B.~{Barber}, D.~{K{\"u}hn}, A.~{Lo}, D.~{Osthus}, and A.~{Taylor},
  \emph{{Clique decompositions of multipartite graphs and completion of Latin
  squares}}, ArXiv e-prints (2016).

\bibitem[DH84]{DH}
David~E. Daykin and Roland H{\"a}ggkvist, \emph{Completion of sparse partial
  {L}atin squares}, Graph theory and combinatorics ({C}ambridge, 1983),
  Academic Press, London, 1984, pp.~127--132. \MR{777169}

\bibitem[EGv{\etalchar{+}}14]{EGHKPS}
K.~{Edwards}, A.~{Gir{\~a}o}, J.~{van den Heuvel}, R.~J. {Kang}, G.~J. {Puleo},
  and J.-S. {Sereni}, \emph{{Extension from Precoloured Sets of Edges}}, ArXiv
  e-prints (2014).

\bibitem[Eva60]{Ev60}
Trevor Evans, \emph{Embedding incomplete latin squares}, Amer. Math. Monthly
  \textbf{67} (1960), 958--961. \MR{0122728}

\bibitem[Fia03]{fiala2003np}
Ji{\v{r}}{\'\i} Fiala, \emph{Np completeness of the edge precoloring extension
  problem on bipartite graphs}, Journal of Graph Theory \textbf{43} (2003),
  no.~2, 156--160.

\bibitem[Gal95]{Ga95}
Fred Galvin, \emph{The list chromatic index of a bipartite multigraph}, J.
  Combin. Theory Ser. B \textbf{63} (1995), no.~1, 153--158. \MR{1309363}

\bibitem[GK16]{GK}
A.~{Gir{\~a}o} and R.~J. {Kang}, \emph{{Precolouring extension of Vizing's
  theorem}}, ArXiv e-prints (2016).

\bibitem[H{\"a}g78]{H78}
R.~H{\"a}ggkvist, \emph{A solution of the {E}vans conjecture for {L}atin
  squares of large size}, Combinatorics ({P}roc. {F}ifth {H}ungarian {C}olloq.,
  {K}eszthely, 1976), {V}ol. {I}, Colloq. Math. Soc. J\'anos Bolyai, vol.~18,
  North-Holland, Amsterdam-New York, 1978, pp.~495--513. \MR{519287}

\bibitem[H{\"a}g89]{Ha89}
Roland H{\"a}ggkvist, \emph{A note on {L}atin squares with restricted support},
  Discrete Math. \textbf{75} (1989), no.~1-3, 253--254, Graph theory and
  combinatorics (Cambridge, 1988). \MR{1001400}

\bibitem[HC92]{haggkvist1992some}
Roland H{\"a}ggkvist and Amanda Chetwynd, \emph{Some upper bounds on the total
  and list chromatic numbers of multigraphs}, Journal of graph theory
  \textbf{16} (1992), no.~5, 503--516.

\bibitem[K{\"o}n16]{konig1916graphen}
D{\'e}nes K{\"o}nig, \emph{{\"U}ber graphen und ihre anwendung auf
  determinantentheorie und mengenlehre}, Mathematische Annalen \textbf{77}
  (1916), no.~4, 453--465.

\bibitem[MS90]{MS90}
O.~Marcotte and P.~D. Seymour, \emph{Extending an edge-coloring}, J. Graph
  Theory \textbf{14} (1990), no.~5, 565--573. \MR{1073098}

\bibitem[Sme81]{S81}
Bohdan Smetaniuk, \emph{A new construction on {L}atin squares. {I}. {A} proof
  of the {E}vans conjecture}, Ars Combin. \textbf{11} (1981), 155--172.
  \MR{629869}

\bibitem[Viz64]{vizing1964estimate}
Vadim~G Vizing, \emph{On an estimate of the chromatic class of a p-graph},
  Diskret. Analiz \textbf{3} (1964), no.~7, 25--30.

\end{thebibliography}

\end{document}